\documentclass[12pt]{article}
\usepackage{geometry}

\usepackage{amsfonts}
\usepackage{amsmath}
\usepackage{amssymb}
\usepackage{amsthm}

\usepackage[dvipdfmx]{graphicx}

\newcommand{\Arg}{{\rm Arg}}

\newcommand{\Set}[1]{\left\{ #1 \right\}}
\newcommand{\Z}{{\mathbb Z}}

\newcommand{\KK}{K}
\newcommand{\barz}{{\bar z}}

\newcommand{\calF}{{\cal F}}

\newcommand{\calV}{{\cal V}}
\newcommand{\cx}{{\mathbb C}}

\newcommand{\fdom}{D}
\newcommand{\fracpart}[1]{\langle #1 \rangle}

\newcommand{\grad}{{\rm grad}}
\newcommand{\halfplane}{{\mathbb H}}
\newcommand{\half}{\frac{1}{2}}
\newcommand{\id}{{\rm id}}
\newcommand{\im}{{\rm Im}}
\newcommand{\interior}{{\rm int}}
\newcommand{\punccx}{{\mathbb C^\times}}
\newcommand{\rat}{{\mathbb Q}}
\newcommand{\real}{{\mathbb R}}
\newcommand{\re}{{\rm Re}}

\newcommand{\rme}{{\rm e}}
\newcommand{\rmi}{{\sqrt{-1}}}
\newcommand{\set}[1]{\{ #1 \}}
\newcommand{\spiral}{S}
\newcommand{\tildez}{z}
\newcommand{\uu}{u}
\newcommand{\vv}{v}
\newcommand{\ww}{w}
\newcommand{\zahl}{{\mathbb Z}}

\newtheorem{theorem}{Theorem}
\newtheorem{corollary}[theorem]{Corollary}

\newtheorem{lemma}[theorem]{Lemma}
\newtheorem{proposition}[theorem]{Proposition}

\title{Area convergence of Voronoi cells on spiral lattices}

\author{Yoshikazu Yamagishi, Takamichi Sushida,
  and Jean-Fran\c{c}ois Sadoc}

\begin{document}

\maketitle

\begin{abstract}
It is shown that the area of Voronoi cells for a generalized
Archimedean spiral lattice converges under some scale normalization,
if the angle parameter is badly approximable.
\end{abstract}

\medskip


\section{Introduction}

Phyllotaxis is a term coined by botanists that describes the
organisation of florets in flowers (like daisy) or positions of leaves
along stems of trees or plants.  For a physicist, phyllotaxis can be
seen as the result of a maximum packing of disks on a plane surface.
If disks have some softness and the packing is confined within a
circular border, a phyllotactic solution can give a best packing
fraction \cite{sadoc_2012_a}.  Numerical analysis on the areas of
Voronoi cells of phyllotactic patterns have suggested a large
homogeneity of these areas \cite{sadoc_2013}.  The aim of this paper
is to give a mathematical proof for the convergence of the area of
Voronoi cells for Fermat and Archimedean spiral lattices.

Let $\alpha>0$, $\theta\in \real$.  
 The set 
\[
 \spiral = \spiral(\alpha,\theta) := \set{z_{j} := j^\alpha
  \rme^{j\theta\rmi} : j\in\zahl_{\ge0}} 
\]
is called a (generalized) Archimedean spiral lattice with an exponent
$\alpha$.  It is called a Fermat spiral lattice if $\alpha= \half$.
In botany, the angle parameter $\theta$ is called `divergence'
\cite{coxeter}.  The Voronoi cell of $z_{j}$ in $\spiral$ is given by
\[
V(z_{j}, \spiral)
= \set{\zeta \in \cx:  |\zeta - z_{j}| \le |\zeta - z'|,
 \forall z' \in \spiral}.
\]
The Voronoi tessellation for the Fermat spiral lattice $\alpha=\half$
has been studied as a geometric model of spiral phyllotaxis
\cite{dixon, hotton}.

Sadoc et al. \cite{sadoc_2013} studied the Fermat spiral lattice with
the golden section divergence angle $\theta= 2\pi\tau$, $\tau = \frac{1+\sqrt{5}}{2}$,
and observed the convergence of the area $|V(z_{j},\spiral)|$ as
$j\to\infty$, as well as the existence of the grain boundary, its
quasiperiodicity and self-similarity under the substitution rule.
Yamagishi et al.~\cite{yamagishi_2018_a} proved the existence and the
quasiperiodicity of the grain boundary for the case of the Archimedean
spiral lattice $\alpha=1$, by considering the topological suspension. 
The case for general $\alpha$ is addressed in
\cite{gar}.
See Figure \ref{Figure_1_1} for the Voronoi tessellations on $S = S(\alpha, 2\pi\tau)$, with $\alpha = 1, \frac{1}{2}, \frac{1}{4}$.

This paper shows the following theorem.  An irrational number is
called badly approximable if its continued fraction expansion has
bounded partial quotients.

\begin{theorem}
  \label{18}
  Suppose that $\theta/2\pi$ is badly approximable.  Then
  \[
  \lim_{j\to\infty} j^{1-2\alpha}|V(z_{j},\spiral)| = 2\pi\alpha.
  \]
\end{theorem}

\begin{corollary}[Area convergence]
\label{18_1}
  Suppose that $\theta/2\pi$ is badly approximable.  
  Then we have
  $\lim_{j\to\infty} |V(z_{j},\spiral)| = \infty$ if $\alpha > \half$,
  $\lim_{j\to\infty} |V(z_{j},\spiral)| = 0$ if $\alpha < \half$,
  and 
  \[
  \lim_{j\to\infty} |V(z_{j},\spiral)| = \pi
  \]
  if $\spiral$ is a Fermat spiral lattice $\alpha= \half$.
\end{corollary}

Figure \ref{Figure_1_2} shows the sequences of areas $|V(z_j,S)|$ 
for $S=S(\alpha,2\pi\tau)$, $\tau=(1+\sqrt{5})/2$, $j=1,2,\dots,10000$,
where $\alpha =1, \half, \frac{1}{4}$.
Figure \ref{Figure_1_3} observes the convergence $j^{1-2\alpha}|V(z_{j},\spiral)| \to 2\pi\alpha$ for 
$\alpha = 1, \half, \frac{1}{4}$.
In particular, in the Fermat spiral lattice $\alpha= \half$, the area $|V(z_{j},\spiral)|$ tends to $\pi$.

The outline of this paper is as follows.  Section 2 is a short
introduction to continued fractions and Farey intervals
\cite{hardywright}.  Section 3 recalls the geometry of Voronoi
tessellations on linear lattices $\Lambda = \Lambda(z) = \zahl +
z\zahl$, \cite{hellwig,helical}.  Here we adopt the botanical term
`opposed parastichy pair' and introduce Richards' formula \cite{jean}.
The `rise in phyllotaxis' in botany, or the parastichy transition
\cite{rothen_koch_1989_a}, considers a geodesic flow on the modular
surface $SL(2,\zahl) \backslash \halfplane$, in the vertical direction
by fixing the `divergence' $\re(z)$ of the parameter $z \in
\halfplane$ and taking the limit as $\im(z) \to 0$.  It is a `scenery
flow' on the geometric structures on the torus \cite{arnoux-fisher}.

Section 4 is a technical part.  A basis $\lambda_1, \lambda_2$ of a
linear lattice $\Lambda = \lambda_1 \zahl + \lambda_2 \zahl$ is called
{\em reduced} if $\lambda_2 / \lambda_1 \in \fdom$, where
$\fdom = \set{\zeta \in \halfplane :
  |\zeta|\ge1, -\half \le \re(\zeta) \le \half}$.
It is shown that if $\Lambda$ is a linear lattice with a reduced basis
$\lambda_1, \lambda_2$, and $\phi : \Lambda \to \cx$ is a perturbation
with $|\phi-\id|$ small, then the Voronoi cell $V(\phi(0),
\phi(\Lambda))$ is adjacent to the cell $V(\phi(\lambda),
\phi(\Lambda))$, in the perturbed nonlinear lattice, only if either
$\lambda = \pm \lambda_1, \pm \lambda_2, \pm \lambda_1 \pm \lambda_2$.
This implies that the area $|V(\phi(0),\phi(\Lambda))|$ is close to
the area $|V(0,\Lambda)|$.  Here we assume that $\im(\lambda_2 /
\lambda_1)$ is bounded from above.

The adjacency of Voronoi cells depends on the bifurcation of the
vertices of Voronoi tessellations.  A vertex of Voronoi cell is a
circumcenter of a triangle of sites in $\Lambda$.  If we perturb a
corner of a triangle, then the circumcenter stays on the perpendicular
bisector of the edge for the other two corners of the triangle.
 
Section 5 is the main part of this paper.  Instead of treating the
discrete family of Voronoi cells, we consider a suspension flow in the
space of Archimedean spiral lattices.  Locally, it can be regarded as
a perturbation of a geodesic flow in the space of linear lattices.  By
Taylor's theorem, the area of the Voronoi cell in the nonlinear
lattice is close to the area of a Voronoi cell in the linear lattice,
which shows Theorem~\ref{18}.

See \cite{pennybacker,barabe} for the history of the study of
phyllotaxis.  The circle packing model of spiral phyllotaxis is
addressed in \cite{rothen_koch_1989_b,yamagishi_2017_a}.  It has been
shown only recently that a Fermat spiral lattice is a Delone set
(relatively dense and uniformly discrete) if and only if its
divergence angle is badly approximable in \cite{akiyama}.  See also
\cite{marklof}.

\begin{figure}
\centering
(a)
\includegraphics[scale=.6]{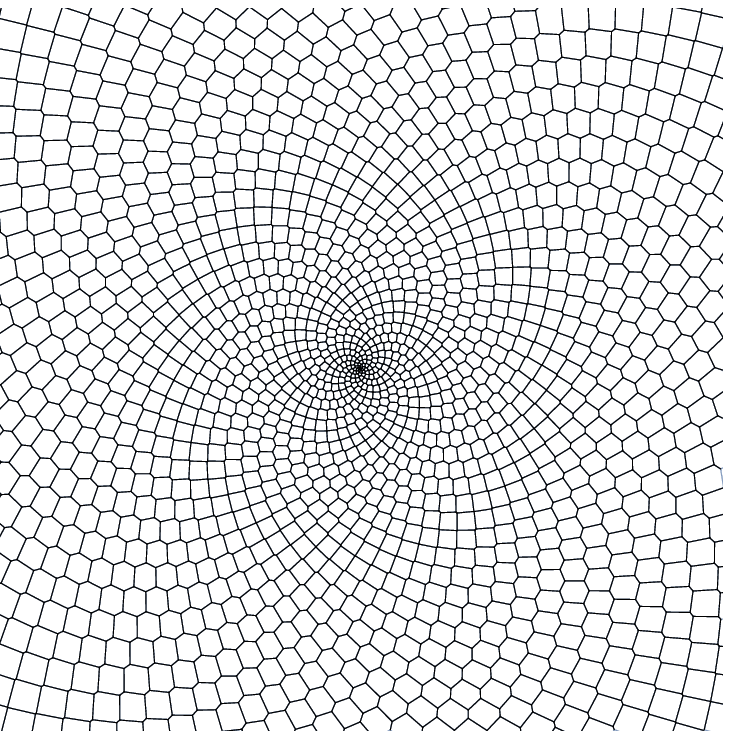}
(b)
\includegraphics[scale=.6]{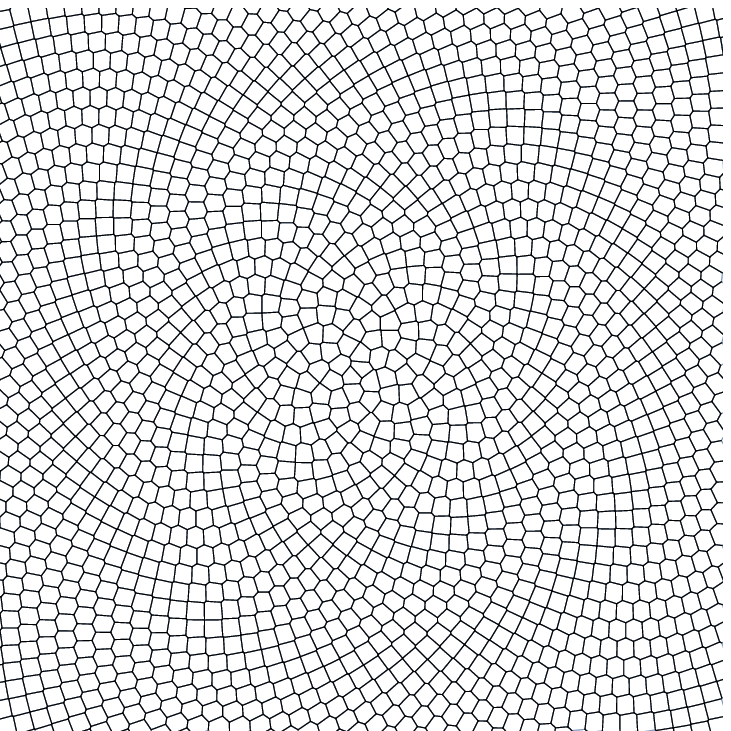}

%\vspace{1zw}
(c)
\includegraphics[scale=.6]{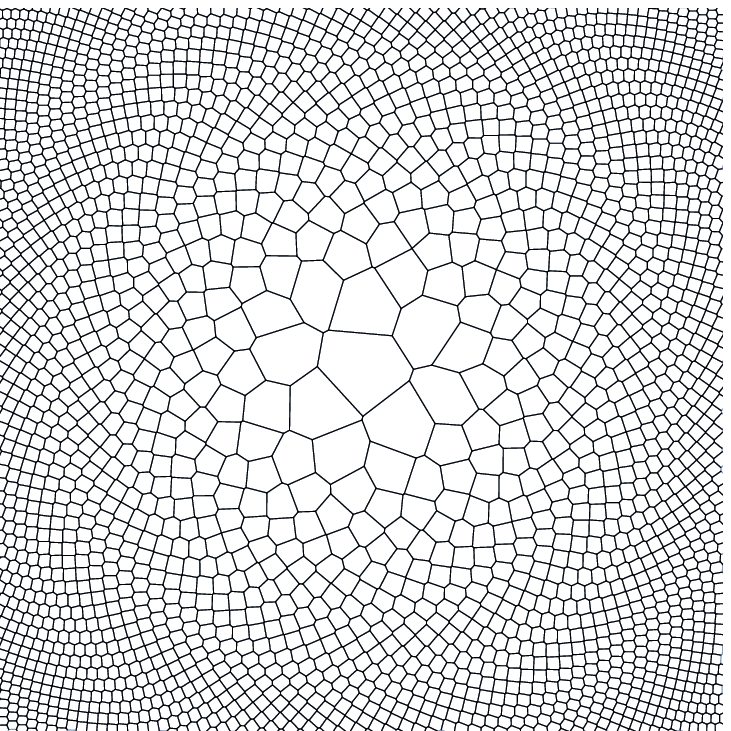}
\caption{Voronoi tessellations for spiral lattices $S = S(\alpha,\theta)$.
(a) The Archimedean spiral lattice with the exponent $\alpha=1$, divergence angle $\theta=2\pi\tau$, where $\tau = (1+\sqrt{5})/2$.  Plot range: $[-1296,1296] \times [-1296,1296]$.
(b) The Fermat spiral lattice $\alpha=1/2$, $\theta=2\pi\tau$.  Plot range: $[-36,36]\times[-36,36]$.
(c) $\alpha=1/4$, $\theta=2\pi\tau$.  Plot range: $[-6,6]\times[-6,6]$.
}
\label{Figure_1_1}
\end{figure}

\begin{figure}
\centering
(a)
\includegraphics[scale=.5]{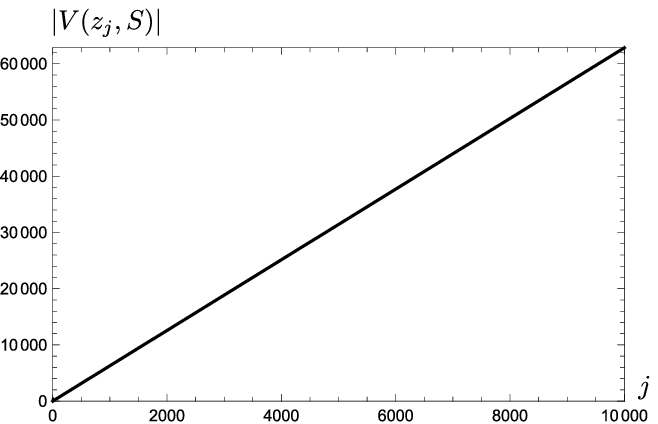}
(b)
\includegraphics[scale=.5]{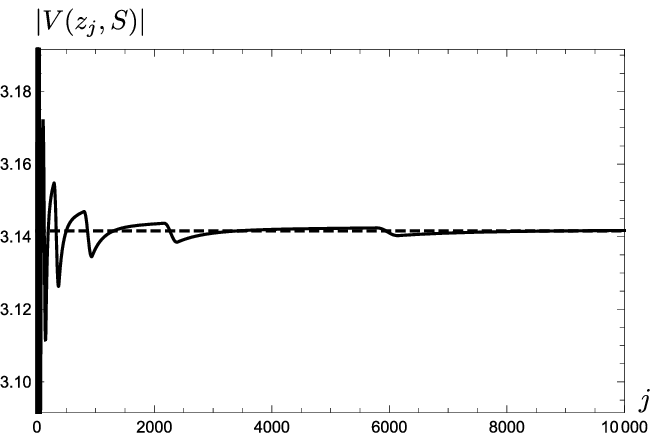}
(c)
\includegraphics[scale=.5]{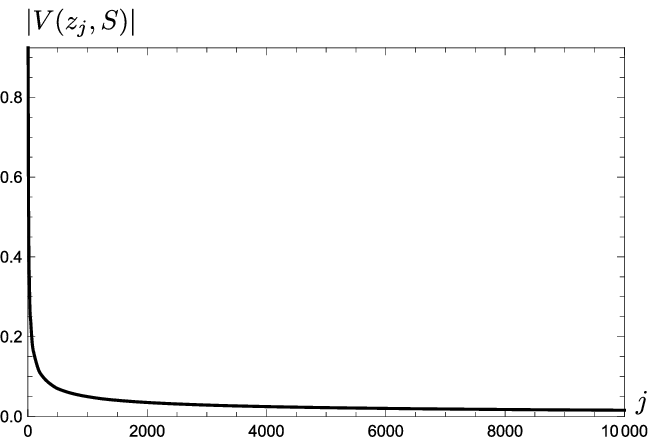}
\caption{
Areas of Voronoi cells $|V(z_j, S)|$, $j=1,2,\dots,10000$.
(a) The Archimedean spiral lattice $S = S(1,2\pi\tau)$, $\tau = (1+\sqrt{5})/2$.
(b) The Fermat spiral lattice $S(1/2, 2\pi\tau)$. 
(c) $S(1/4, 2\pi\tau)$. 
}
\label{Figure_1_2}
\end{figure}

\begin{figure}
\centering
(a)
\includegraphics[scale=.5]{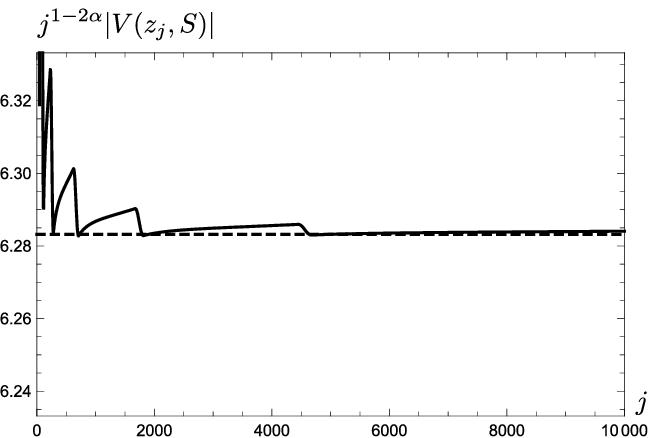}
(b)
\includegraphics[scale=.5]{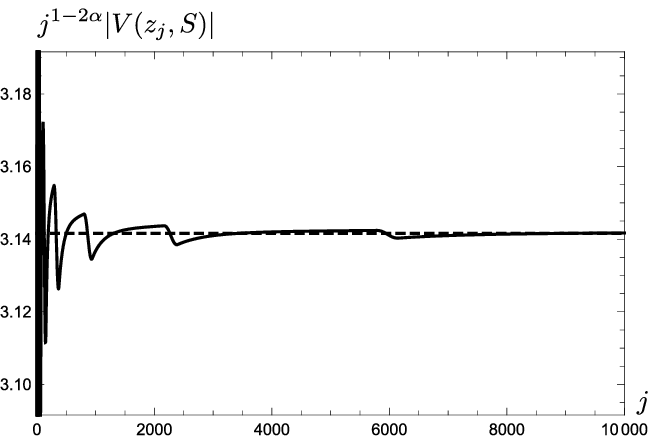}
(c)
\includegraphics[scale=.5]{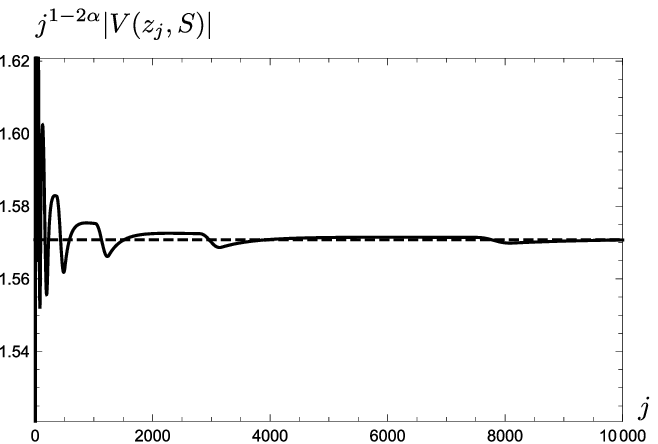}
\caption{
Areas of Voronoi cells, normalized, $j^{1-2\alpha}|V(z_j, S)|$, $1 \le j \le 10000$.
(a) The Archimedean spiral lattice $S(1, 2\pi\tau)$.
(b) The Fermat spiral lattice $S(1/2, 2\pi\tau)$. 
(c) $S(1/4, 2\pi\tau)$. 
}
\label{Figure_1_3}
\end{figure}

\section{Farey intervals}

Here we recall some basic properties of the Farey series and continued
fractions.  A fraction $\frac{p}{q} \in \rat$ always assumes that
$p\in\zahl$ and $q \in \zahl_{>0}$.  Let $k\in\zahl_{>0}$.  The set of
fractions with denominator $\le k$, $\calF_k = \set{\frac{p}{q} \in
  \rat : p,q \in \zahl, \ 1 \le q \le k}$, is called the \textit{Farey
  series} of order $k$.  We have $\calF_1 = \zahl \subset \calF_2
\subset \calF_3 \subset \dots$, and $\bigcup_{k} \calF_k = \rat$.  An
open interval $(\frac{a}{m}, \frac{b}{n})$ with irreducible fractions
$\frac{a}{m}, \frac{b}{n}$ is called a \textit{Farey interval} of
order $k \in \zahl_{>0}$ if it is a connected component of $\real
\setminus \calF_k$.  A pair of fractions $\frac{a}{m} < \frac{b}{n}$
is called a \textit{Farey pair} if $mb-na=1$.

\begin{lemma} 
\label{36}
A pair of irreducible fractions $\frac{a}{m} < \frac{b}{n}$ is a Farey
pair if and only if the interval $(\frac{a}{m}, \frac{b}{n})$ is a
Farey interval of some order $k \in \zahl_{>0}$.  If it is the case,
we have $\max\set{m,n} \le k < m+n$ and
\begin{equation}
\label{1}
\calF_{m+n} \cap
\left(\frac{a}{m}, \frac{b}{n} \right) = \Set{\frac{a+b}{m+n}}.
\end{equation}
\end{lemma}
\begin{proof}
Let $\frac{a}{m} < \frac{b}{n}$ be a Farey pair, and suppose that
$\frac{a}{m} < \frac{p}{q} < \frac{b}{n}$.  Since $q\frac{a}{m}- p =
\frac{qa-pm}{m}<0$ and $q\frac{b}{n}- p = \frac{qb-pn}{n}>0$, we have
$q\frac{a}{m}- p \le -\frac{1}{m}$ and $q\frac{b}{n}- p \ge
\frac{1}{n}$.  Thus
\[
  q =
  \frac{(q\frac{b}{n} - p) - (q\frac{a}{m}- p)}{\frac{b}{n} - \frac{a}{m}}
  \ge \frac{\frac{1}{n} + \frac{1}{m}}{\frac{b}{n} - \frac{a}{m}}
  = m+n.
\]
Moreover we have
\[ \frac{a+b-1}{m+n}
  \le \frac{a}{m}
  < \frac{a+b}{m+n}
  < \frac{b}{n}
  \le \frac{a+b+1}{m+n},
\]
which implies (\ref{1}).

For the converse, let $\frac{a}{m} \in \calF_k$.  Let $b,n \in \zahl$
be a (unique) solution of the equation
\[ mb - na = 1
\]
under the condition $k - m < n \le k$.  Then we have $\max\set{m,n}
\le k < m+n$.  The argument above shows that $(\frac{a}{m},
\frac{b}{n})$ is a Farey interval of order $k$.  See also
\cite[Theorem 28]{hardywright}.
\end{proof}

Let
\begin{equation}
\label{25}
 x
  = a_0+\frac{1}{a_1+\frac{1}{a_2+\cdots}} = [a_0,a_1,a_2,\cdots],
  \ a_0 \in \Z, \ a_i \in \Z_{>0}, \ i\geq 1
\end{equation}
be a continued fraction expansion of $x \in \real$.  We often assume
that $x$ is irrational.  Define the sequences $\set{p_i}_{i\geq -1}$
and $\set{q_i}_{i\geq -1}$ by $p_{-1} = 1$, $q_{-1} = 0$, $p_{0}=
a_0$, $q_{0}=1$, $p_{1}=a_0 a_1 + 1$, $q_{1}=a_1$, and $p_{i+1} =
a_{i+1} p_{i} + p_{i-1}$, $q_{i+1} = a_{i+1} q_{i} + q_{i-1}$,
$i\ge1$.  Let $p_{i,k} = k p_{i} + p_{i-1}$, $q_{i,k} = k q_{i} +
q_{i-1}$ for $i \ge0$, $0 \le k \le a_{i+1}$.  Note that $p_{i,0} =
p_{i-1}$, $q_{i,0} = q_{i-1}$, $p_{i,a_{i+1}} = p_{i+1}$,
$q_{i,a_{i+1}} = q_{i+1}$.  The fraction ${p_i}/{q_i} =
[a_0,a_1,\cdots,a_i]$, $i\ge0$, is called a {\it (principal)
  convergent} of $x$, and ${p_{i,k}}/{q_{i,k}} = [a_0, a_1, \cdots,
  a_i, k]$, $i\ge0$, $0<k<a_{i+1}$, is called an {\it intermediate
  convergent} of $x$.  An induction shows that if $i$ is odd and $0
\le k \le a_{i+1}$, then $(\frac{p_{i,k}}{q_{i,k}}, \frac{p_i}{q_i})$
is a Farey interval and $\frac{p_{i,k}}{q_{i,k}} < x <
\frac{p_i}{q_i}$.  If $i$ is even, $0 \le k \le a_{i+1}$ and
$(i,k)\neq(0,0)$, then $(\frac{p_i}{q_i}, \frac{p_{i,k}}{q_{i,k}})$ is
a Farey interval and $\frac{p_i}{q_i} < x < \frac{p_{i,k}}{q_{i,k}}$.

\begin{lemma} 
\label{38}
Let $(\frac{a}{m}, \frac{b}{n})$ be a Farey interval containing $x$.
Then we have either $(\frac{a}{m}, \frac{b}{n}) =(\frac{p_i}{q_i},
\frac{p_{i,k}}{q_{i,k}})$ for some $i$ even and $0 \le k < a_{i+1}$,
or $(\frac{a}{m}, \frac{b}{n}) =(\frac{p_{i,k}}{q_{i,k}},
\frac{p_i}{q_i})$ for some $i$ odd and $0 \le k < a_{i+1}$.
\end{lemma}
\begin{proof}
The interval $(\frac{a}{m}, \frac{b}{n})$ is a connected component of
$\real\setminus\calF_{\max\set{m,n}}$.  Since $\set{q_i}_i$ is an
increasing sequence, there exists $i$ such that $q_i \le \max\set{m,n}
< q_{i+1}$.  If $q_i \le \max\set{m,n} < q_{i,1}$, then the connected
component of $\real\setminus\calF_{\max\set{m,n}}$ containing $x$ is
written as $(\frac{p_{i-1}}{q_{i-1}}, \frac{p_i}{q_i})$ or
$(\frac{p_{i}}{q_{i}}, \frac{p_{i-1}}{q_{i-1}})$.  Otherwise, there
exists $1\le k< a_{i+1}$ such that
$q_{i,k} \le \max\set{m,n} < q_{i,k+1}$,
so the connected component of
$\real\setminus\calF_{\max\set{m,n}}$ containing $x$ is written as
$(\frac{p_{i,k}}{q_{i,k}}, \frac{p_i}{q_i})$ or $(\frac{p_{i}}{q_{i}},
\frac{p_{i,k}}{q_{i,k}})$.
\end{proof}

A fraction $\frac{a}{m}$ is called an approximation to $x\in\real$
on the left, if $\frac{a}{m} < x$ and $0 < mx-a \le qx-p$ whenever
$\frac{p}{q} \le x$ and $0<q\le m$.  A fraction $\frac{b}{n}$
is called an approximation to $x$ on the right, if $\frac{b}{n} > x$
and $0 < b - nx \le p - qx$ whenever $\frac{p}{q} \ge x$ and
$0<q\le n$.  It is clear that a fraction is irreducible
if it is an approximation to some $x$ (on the left or right).

\begin{lemma} 
  \label{39}
  Let $x \in \real \setminus \zahl$.  A fraction is an
  approximation to $x$ on the left or right, if and only if it is
  an endpoint of a Farey interval containing $x$.
\end{lemma}

\begin{proof}
  First, we suppose that $\frac{a}{m} < x$ is an approximation to
  $x$ on the left, and show that it is an endpoint of a Farey interval
  containing $x$.  For any fraction $\frac{p}{q} < x$ with $0<p\le
  m$, we have $0 < mx-a \le qx-p$.  This implies that $x -
  \frac{a}{m} \le x - \frac{p}{q}$, and $\frac{p}{q} \le
  \frac{a}{m}$.  So the connected component of $\real \setminus
  \calF_{m}$ containing $x$ has an endpoint $\frac{a}{m}$.  The proof
  for the case $\frac{b}{n} > x$ is similar.

  Conversely, assume that $(\frac{a}{m}, \frac{b}{n})$ is a Farey
  interval containing $x$, and we will show that $\frac{a}{m}$ is an
  approximation to $x$ on the left.  Let $\frac{p}{q}$ be a fraction
  such that $\frac{p}{q} < \frac{a}{m}$, and suppose that $0 < q\le
  m$.  If $q=m$, then we have $p<a$ and $mx-a < qx-p$.  If $q<m$, then
  we have
  \[ \frac{a-p}{m-q} - \frac{b}{n}
   = \frac{qb-np-1}{(m-q)n} \ge 0,
  \]
  which implies that $\frac{a-p}{m-q} \ge \frac{b}{n} > x$ and $mx-a <
  qx-p$.  So $\frac{a}{m}$ is an approximation to $x$ on the left.  We
  can similarly show that $\frac{b}{n}$ is an approximation to $x$ on
  the right.
\end{proof}

%%%%%%%%%%%%%%%%%%%%%%%%%%%%%%%%%%%%%%%%%%%%
\section{Voronoi tessellation for a linear lattice}

Here we introduce some basic properties of Voronoi tessellations for
linear lattices, including Richards' formula, by using the botanical
term `parastichy.'  Let $\Lambda = \Lambda(z) := z\zahl + \zahl$,
where $z = x + \rmi y \in \cx$, $x>0$.  The Voronoi cell of the site
$\lambda \in \Lambda$ is given by
\[
V(\lambda) = V(\lambda, \Lambda)
= \set{\zeta \in \cx:  |\zeta - \lambda| \le |\zeta-\lambda'|,
 \forall \lambda' \in \Lambda}.
\]
The family of Voronoi cells $\calV = \set{V(\lambda)}_{\lambda \in
  \Lambda(z)}$ is a tiling of the plane, that is, $\cx =
\bigcup_{\lambda \in \Lambda} V(\lambda)$, and $\interior(V(\lambda))
\cap \interior(V(\lambda')) = \emptyset$ for $\lambda \neq \lambda'$.

The line segment $\ell(\lambda, \lambda')$ with the endpoints
$\lambda, \lambda' \in \Lambda$, is called a Delaunay edge if it
satisfies the Empty Circumdisk Property, that is, there exists a disk
$K$ such that $K \cap \Lambda = \set{\lambda, \lambda'}$.  For
distinct points $w_1, w_2, w_3 \in \cx$, let $\angle(w_1, w_2, w_3) :=
\Arg(\frac{w_1-w_2}{w_3-w_2})$, where $-\pi < \Arg \le \pi$ denotes
the principal argument.

\begin{lemma}
\label{33}
  Let $\lambda, \lambda' \in \Lambda$ be distinct sites.  The
  following conditions are mutually equivalent.
  \begin{enumerate}
  \item The Voronoi cells $V(\lambda)$, $V(\lambda')$ are
    (edge-)adjacent.
  \item The line segment $\ell(\lambda, \lambda')$ is a Delaunay edge.
  \item $\ell(\lambda, \lambda') \cap \Lambda = \set{\lambda,
    \lambda'}$, and we have
    \[ \angle(\lambda, w, \lambda')
    + \angle(\lambda', w', \lambda) < \pi
    \]
    for any $w, w' \in \Lambda$ whenever $0< \angle(\lambda, w,
    \lambda') <\pi$ and $0< \angle(\lambda', w', \lambda) < \pi$.
  \end{enumerate}
\end{lemma}
\begin{proof}
  See \cite[Lemma 1]{yamagishi_2015_a}.
\end{proof}

In this paper a line segment $\ell(\lambda, \lambda')$ is not called a
Delaunay edge when $V(\lambda)$ is only corner-adjacent to
$V(\lambda')$.  The Delaunay edges do not intersect each other.  A
connected component of the complement of the union of the Delaunay
edges is called a Delaunay polygon.  A Delaunay polygon for the linear
lattice $\Lambda(z)$ is either a rectangle or an acute triangle.

An irreducible fraction $\frac{a}{m}$, $a\in\zahl$, $m\in\zahl_{>0}$,
is called a parastichy index of $\Lambda(z)$ if $V(mz-a)$ is
(edge-)adjacent to $V(0)$.  A pair of parastichy indices $\frac{a}{m},
\frac{b}{n}$ of $\Lambda(z)$ is called an opposed parastichy pair if
$(mx-a)(nx-b)<0$.

\begin{lemma}
  Let $\frac{a}{m}$ be a parastichy index of $\Lambda(z)$.  Then
  $\frac{a}{m}$ is an approximation to $x = \re(z)$ on the left or
  right, or equal to $x$.
\end{lemma}
\begin{proof}
  First suppose that $mx-a>0$.  If $\frac{a}{m}$ is not an
  approximation to $x$ on the left, then there exist $p,q\in \zahl$
  such that $0< q \le m$, and $0< qx-p < mx-a$.  Then we have
  $\frac{\pi}{2}<|\angle(0, qz-p, mz-a)|,
  |\angle(mz-a,(m-q)z-(a-p),0)|$.  By Lemma~\ref{33}, $V(mz-a)$ is not
  adjacent to $V(0)$.  The case $mx-a<0$ is similar.
\end{proof}

\begin{lemma} 
  \label{12}
  Suppose that the cell $V(0)$ is adjacent to $V(mz-a)$ and $V(nz-b)$.
  Then $|mb-na| = 0, 1$.
\end{lemma}
\begin{proof}
  Let $\lambda_1 = mz-a$, $\lambda_2 = nz-b$.  We assume $\lambda_1
  \neq \pm \lambda_2$, and show that $mb-na = \pm1$.  The line
  segments $\ell(\lambda, \lambda+ \lambda_1)$, $\ell(\lambda,
  \lambda+ \lambda_2)$ are Delaunay edges for all $\lambda \in
  \Lambda(z)$.  Since Delaunay edges do not intersect, the
  parallelogram $\square(\lambda, \lambda+\lambda_1,
  \lambda+\lambda_1+\lambda_2, \lambda+\lambda_2)$ has no points of
  $\Lambda$ in its interior, or on the edges, except the corners
  $\lambda, \lambda+\lambda_1, \lambda+\lambda_2,
  \lambda+\lambda_1+\lambda_2$, for any $\lambda \in \Lambda$.  So we
  have $\Lambda(z) = z\zahl + \zahl = \lambda_1 \zahl + \lambda_2
  \zahl$.  This implies that the matrix $\begin{pmatrix} m & n \\ a &
    b \end{pmatrix}$ has an inverse matrix with integer components, so
  $|mb-na| = 1$.
\end{proof}

\begin{lemma}
  Let $\frac{a}{m} < \frac{b}{n}$ be an opposed parastichy pair of
  $\Lambda(z)$.  Then we have either $a=p_i$, $m=q_i$, $b=p_{i,k}$,
  $n=q_{i,k}$ for some $i$ even or $a=p_{i,k}$, $m=q_{i,k}$, $b=p_i$,
  $n=q_i$ for some $i$ odd, and $0<k\le a_{i+1}$.
\end{lemma}
\begin{proof}
  The proof follows from Lemma~\ref{38}, since $(\frac{a}{m},
  \frac{b}{n}) \ni x$ is a Farey interval by Lemma~\ref{12}.
\end{proof}

Note that $q_i z - p_i = \fracpart{q_i x} + q_i y \rmi$, $q_{i,k} -
p_{i,k} = \fracpart{q_{i,k} x} + q_{i,k} y \rmi$ for $i\ge0$, $0 \le k
\le a_{i+1}$.  Since $q_{i,k+1} > q_{i,k}$ and
$|\fracpart{q_{i,k+1}x}| < |\fracpart{q_{i,k}x}|$, we have
\begin{equation}
\label{30}
 |\angle(q_{i}z-p_i ,0, q_{i,k+1}z-p_{i,k+1})|
   < |\angle( q_{i}z-p_i, 0, q_{i,k}z-p_{i,k})|
\end{equation}
for $i\ge0$, $0\le k < a_{i+1}$.

\begin{lemma}
$|q_{i,k}z-p_{i,k}| > |q_iz - p_i|$ for $i\ge0$, $0<k<a_{i+1}$.
\end{lemma}
\begin{proof}
Since
\begin{equation}
\label{71}
\frac{q_{i,k}}{q_i} = k + \frac{1}{a_j + \frac{1}{a_{j-1}+\dots}} > 1,
\quad
-\frac{\fracpart{q_{i,k}x}}{\fracpart{q_ix}} = a_{j+1}-k+\frac{1}{a_{j+2}+\dots} >1,
\end{equation}
we have
\[
|q_iz-p_i|^2
= \fracpart{q_ix}^2 + q_i^2y^2
< \fracpart{q_{i,k}x}^2 + q_{i,k}^2y^2
= |q_{i,k}z-p_{i,k}|.
\]
\end{proof}

\begin{lemma}
Let $a,b,c,d,t\in\real_{>0}$.  Then $\im\left(
\frac{-c+td\rmi}{a+tb\rmi} \right) \le
\half(\frac{c}{a}+\frac{d}{b})$.
\end{lemma}
\begin{proof}
$\im\left( \frac{-c+td\rmi}{a+tb\rmi} \right) 
= \frac{t(ad+bc)}{a^2+t^2b^2} \le \frac{ad+bc}{2ab}
= \half(\frac{c}{a}+\frac{d}{b})$.
\end{proof}

\begin{lemma}
\label{68}
$\left| \im\left( \frac{q_{i,k}z-p_{i,k}}{q_iz-p_i} \right) \right|
\le 1+\half a_{i+1}$ for $i\ge0$, $0 \le k \le a_{i+1}$,
$(i,k)\neq(0,0)$.
\end{lemma}
\begin{proof}
\begin{align*}
\left| \im\left( \frac{q_{i,k}z-p_{i,k}}{q_iz-p_i} \right) \right| 
&= \left| \im\left( \frac{\fracpart{q_{i,k}x}+q_{i,k}y\rmi}{\fracpart{q_ix}+q_iy\rmi} \right) \right| \\
%& \le \frac{q_{i,k}\fracpart{q_i x} - q_i \fracpart{q_{i,k}x}}{2q_i \fracpart{q_ix}} \\
& \le \half \left( \frac{q_{i,k}}{q_i} - \frac{\fracpart{q_{i,k}x}}{\fracpart{q_ix}} 
\right) \\
&= \half \left( k + \frac{1}{a_i+\dots} 
  + a_{i+1}-k+\frac{1}{a_{i+2}+\dots} \right) \\ 
&\le 1+\half a_{i+1}.
\end{align*}
\end{proof}

\begin{lemma}
  \label{31}
  Let $i\ge0$, $0\le k<a_{i+1}$, $(i,k)\neq(0,0)$.   
  If $\frac{p_{i,k+1}}{q_{i,k+1}}$ is a
  parastichy index for $\Lambda(z)$,
  then $|\angle(q_{i}z-p_i, 0,
  q_{i,k}z - p_{i,k})| > \frac{\pi}{2}$.
\end{lemma}
\begin{proof}
  First suppose that $i$ is even.  Assume that
  $\angle(q_{i,k}z-p_{i,k}, 0, q_{i}z - p_{i}) \le \frac{\pi}{2}$.  In
  the parallelogram $\square(0, q_i z - p_i, q_{i,k+1}z-p_{i,k+1},
  q_{i,k}z-p_{i,k})$, we have
\[ \angle(q_{i,k+1}z-p_{i,k+1}, q_{i,k}z-p_{i,k}, 0) + \angle(0,q_{i}z-p_i,
  q_{i,k+1}z-p_{i,k+1}) > \pi. 
\]
   So $p_{i,k+1}/q_{i,k+1}$ is not a parastichy index.
  The case for odd $i$ is shown by a similar argument.
\end{proof}

\begin{lemma}
  \label{32}
  Let $i\ge0$, $0< k \le a_{i+1}$, $(i,k)\neq(0,1)$.  If
  $\frac{p_{i,k-1}}{q_{i,k-1}}$ is a parastichy index for
  $\Lambda(z)$, then $|\angle(q_{i,k}z-p_{i,k}, 0, q_{i}z-p_i)| <
  \frac{\pi}{2}$.
\end{lemma}
\begin{proof}
  First suppose that $i$ is even.  Assume that $\angle(q_{i}z-p_{i},
  0, q_{i,k}z-p_{i,k}) \ge \frac{\pi}{2}$.  In the parallelogram $P =
  \square(0, q_{i}z-p_i, (q_{i}z-p_i) + (q_{i,k}z-p_{i,k}),
  q_{i,k}z-p_{i,k})$, we have
  \[
  \angle(q_{i,k}z-p_{i,k}, 0, q_{i}z-p_i) +
  \angle(q_{i}z-p_i, (q_{i}z-p_i) + (q_{i,k}z-p_{i,k}), q_{i,k}z-p_{i,k}) >
  \pi.
  \]
  In the translated parallelogram 
  \[ P - (q_{i}z-p_i) =
  \square(-(q_{i}z-p_i), 0 , q_{i,k}z-p_{i,k},  q_{i,k-1}z-p_{i,k-1}),
  \]
  we have
  \[
  \angle(q_{i,k-1}z-p_{i,k-1},
  -(q_{i}z-p_i), 0) + \angle(0, q_{i,k}z-p_{i,k}, q_{i,k-1}z-p_{i,k-1}) > \pi.
  \]
  So $p_{i,k-1}/q_{i,k-1}$ is not a parastichy index.  The case for
  odd $i$ is similar.
\end{proof}

Let
\begin{equation}
  \label{26}
  \eta_{i,k}:=
  \sqrt{-\frac{(q_ix-p_i)(q_{i,k}x-p_{i,k})}{q_iq_{i,k}}}
\end{equation}
where $i\ge0$, $0 \le k \le a_{i+1}$, $(i,k)\neq(0,0)$.  It is called
Richards' formula \cite[p.29]{jean}.  We have
$\eta_{i,a_{i+1}}=\eta_{i+1,0}$ for $i\ge0$, and $\eta_{i,k} >
\eta_{i,k+1}$ for $i\ge0$, $0\le k < a_{i+1}$, $(i,k)\neq(0,0)$.

\begin{lemma}
\label{58}
  Let $z = x + \rmi y$.  Let $i\ge0$, $0\le k < a_{i+1}$,
  $(i,k)\neq(0,0)$.  If $y = \eta_{i,k}$, then the Voronoi cell $V(0)$
  is a rectangle with two parastichy indices $\frac{p_i}{q_i},
  \frac{p_{i,k}}{q_{i,k}}$.  If $\eta_{i,k+1} < y < \eta_{i,k}$, then
  $V(0)$ is a hexagon with three parastichy indices $\frac{p_i}{q_i},
  \frac{p_{i,k}}{q_{i,k}}, \frac{p_{i,k+1}}{q_{i,k+1}}$.
\end{lemma}
\begin{proof}
  First suppose that $y = \eta_{i,k}$.  Then we have
  \[
  (q_{i,k}x-p_{i,k})(q_i x- p_i) + q_{i,k}y \cdot q_i y = 0,
  \]
  and so $|\angle(q_{i,k}z-p_{i,k}, 0, q_{i}z-p_i)| = \frac{\pi}{2}$.  By
  (\ref{30}) and Lemmas~\ref{31}, \ref{32}, a parastichy index is
  either $\frac{p_i}{q_i}$ or $\frac{p_{i,k}}{q_{i,k}}$.

  Next suppose that $\eta_{i,k} > y > \eta_{i,k+1}$.  Then we have
  $|\angle(q_{i,k+1}z-p_{i,k+1},0, q_{i}z-p_i)| < \frac{\pi}{2} <
  |\angle(q_{i,k}z-p_{i,k}, 0, q_{i}z-p_i)|$.  Lemmas~\ref{31},
  \ref{32} imply that a parastichy index is either $\frac{p_i}{q_i}$,
  $\frac{p_{i,k}}{q_{i,k}}$, or $\frac{p_{i,k+1}}{q_{i,k+1}}$.  The
  Voronoi cell $V(0)$ is not a rectangle, so it is a hexagon and has
  three parastichy indices.
\end{proof}

\begin{lemma}
  \label{20}
  The area of the Voronoi cell $V(0)$ is equal to $y=\im(z)$.
\end{lemma}
\begin{proof}
  By Lemma~\ref{12}, the cell $V(0)$ is adjacent to some $V(mz-a)$,
  $V(nz-b)$, with $|mb-na|=1$.  The cells are parallel translates to
  each other, $V(\lambda) = \lambda + V(0)$, so the area of the cells
  $|V(\lambda)| = |V(0)|$ are equal to the area of the parallelogram
  $H = \square(0, mz-a, (m+n)z-(a+b), nz-b)$,
  \[
  |H| = \left| \det \begin{pmatrix} mx -a & nx -b \\ my & ny
  \end{pmatrix} \right|
  = |(mb - na)y| = y.
  \]
\end{proof}

\begin{figure}
\centering
(a)
\includegraphics[scale=0.6]{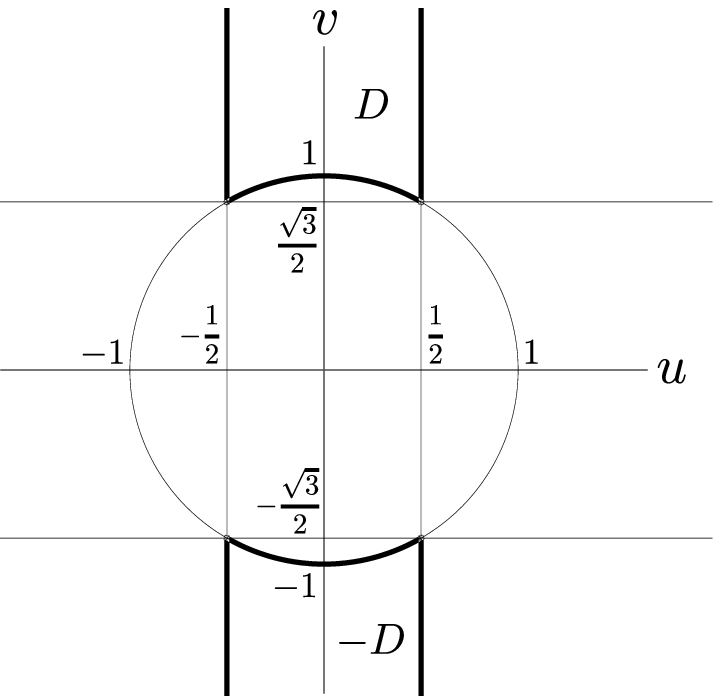}
(b)
\includegraphics[scale=0.6]{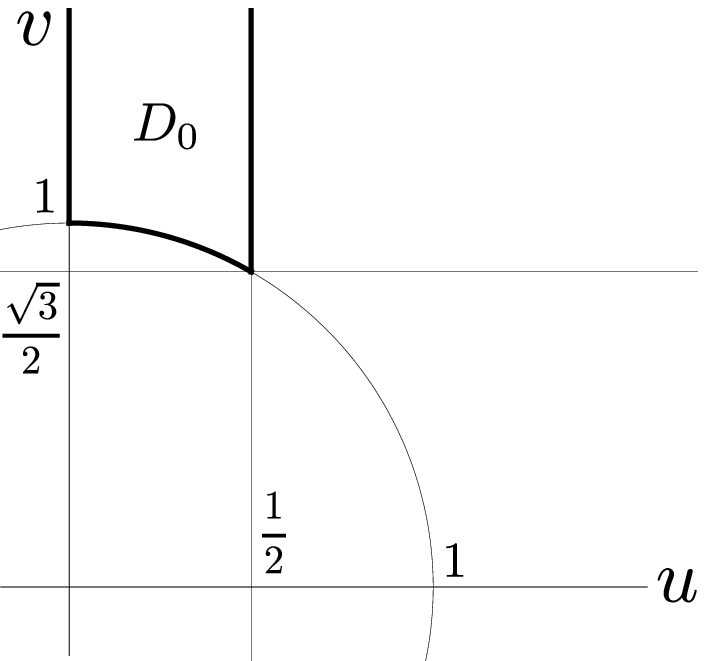}
\caption{(a) The set $\fdom \cup (-\fdom)$ of parameters $\lambda_2 / \lambda_1$
for the pairs of generators $\lambda_1, \lambda_2$ of linear lattices.
(b) The half region $\fdom_0$.}
\end{figure}

We fix $x \in \real$ and suppose that $\eta_{i,k} \ge y > \eta_{i,k+1}$.
There exist $\lambda_1, \lambda_2 
\in \set{ q_iz-p_i, q_{i,k}z-p_{i,k}, q_{i,k+1}z-p_{i,k+1}}$ such that
\[
  \lambda_{2}/\lambda_{1} \in \fdom \cup (-\fdom),
\]
where 
\[
\fdom = \set{w=u+v\rmi \in \cx :
  -\half \le u \le \half, v \ge \frac{\sqrt{3}}{2}, |w|\ge1}.
\]
Denote by
\[
\delta_y := |\lambda_1|,
\]
the length of the shortest vector in $\Lambda(z)$, $z=x+y\rmi$.  The
area of the parallelogram $H = \square(0, \lambda_1, \lambda_1 +
\lambda_2, \lambda_2)$ is equal to $v_y \delta_y^2$, where $v_y :=
|\im(\lambda_2/\lambda_1)|$.  By Lemma~\ref{20} we have $v_y
\delta_y^2 = y$, so
\begin{equation}
\label{72}
\delta_y^2 = \frac{y}{v_y} \le \frac{2}{\sqrt{3}}y.
\end{equation}

%%%%%%%%%%%%%%%%%%%%%%%%%%%%%%%%%%%%%%%%%%%%%%%%%%%%%%%%%%%%%%
\section{Voronoi cells on a perturbation of a linear lattice}

A spiral lattice is locally a perturbation of a linear lattice.  Here
we show that the area of the Voronoi cell in a linear lattice is close
to the area of the Voronoi cell in the perturbed lattice.  A vertex of
a Voronoi tessellation is given as a circumcenter of a triangle.  So
we first consider the perturbation of the circumdisk of a triangle.

\begin{lemma}
\label{50}
Let $f : \cx \to \real$ be a $C^1$ function, and denote by $f(x,y) =
f(x+y\rmi)$ by abuse of notation.  Let $\grad f = (\frac{\partial
  f}{\partial x}, \frac{\partial f}{\partial y})$.  Then we have
$|\grad f|^2 = 4 \frac{\partial f}{\partial z} \frac{\partial
  f}{\partial \bar{z}}$.
\end{lemma}

\begin{proof}
By definition we have $\frac{\partial f}{\partial z} = \half \left(
\frac{\partial f}{\partial x} - \rmi \frac{\partial f}{\partial y}
\right)$, $\frac{\partial f}{\partial \barz} = \half \left(
\frac{\partial f}{\partial x} + \rmi \frac{\partial f}{\partial y}
\right)$.  So we obtain
\[
|\grad f|^2
= \left( \frac{\partial f}{\partial x} \right)^2 + 
\left( \frac{\partial f}{\partial y} \right)^2
=  \left( \frac{\partial f}{\partial z}
 + \frac{\partial f}{\partial \barz}\right)^2 
 + \left( \frac{\partial f}{\partial z}
 - \frac{\partial f}{\partial \barz}\right)^2 
= 4 \frac{\partial f}{\partial z} \frac{\partial f}{\partial \bar{z}}.
\]
\end{proof}

\begin{lemma}
\label{51}
Let $z, z' \in \cx$.  Let $f$ be a real $C^1$ function defined in a
neighborhood of the line segment $\ell(z, z')$.  Suppose that $|\grad
f| \le M_0$ for some $M_0>0$.  Then we have
\[ |f(z') - f(z)|
\le M_0 |z' - z|.
\]
\end{lemma}

\begin{proof}
Let $z(t) = (1-t)z + t z'$, $0 \le t \le 1$.  Then we obtain
\[
 |f(z') - f(z)|
 \le \int_0^1 |\grad f| \; |z'-z|\;dt
 \le M_0 |z'-z|.
\]
\end{proof}

Let $w_0, w_1, w_2 \in \cx$ be distinct.  Assume that
$\im(\frac{w_2-w_0}{w_1-w_0})>0$.  The circumcenter $w_c$ of the
triangle $\triangle(w_0, w_1, w_2)$ is written as
\[
 w_c 
 %% &= \frac{w_1+w_2}{2} + \frac{\rmi(w_2-w_1)}{2} \cot\angle(w_2,w_0,w_1)  \\
  = \frac{w_1+w_2}{2} + \frac{\rmi(w_2-w_1)}{2} f(w_0)
\]
where we denote by
\[ f(w_0)
 := \cot\angle(w_2,w_0,w_1)
 = \frac{\re((\overline{w_1-w_0})(w_2-w_0))}{\im((\overline{w_1-w_0})(w_2-w_0))} 
\]
as a function of $w_0$.
A direct calculation shows that
\[ |\grad f| =
\frac{|(w_2-w_1)(w_1-w_0)(w_2-w_0)|}{(\im((\overline{w_1-w_0})(w_2-w_0)))^2},
\]
which is a symmetric function of $w_0, w_1, w_2$.
%where
%\[ \Delta := \im((\overline{w_1-w_0})(w_2-w_0)) = \det\begin{pmatrix} x_1-x_0 & x_2-x_0 \\ y_1-y_0 & y_2-y_0 \end{pmatrix}.
%\]

Let
\[ \fdom_0 = \set{\ww = \uu + \vv\rmi :
  0 \le \uu \le \half, \vv \ge \frac{\sqrt{3}}{2}, |\ww|\ge1}.
\]
Let $U(z, \epsilon) = \set{\zeta \in\cx : |\zeta - z| < \epsilon}$ be
a neighborhood of $z \in \cx$.

\begin{lemma}
\label{52}
Let $\ww = \uu + \vv\rmi \in \fdom_0$.
Let $w_0 \in U(0,\epsilon)$,
$w_1 \in U(1,\epsilon)$,
$w_2 \in U(\ww, \epsilon)$,
and suppose that $\epsilon \le \frac{1}{50}$.
Then we have
\begin{align*}
  |w_1-w_0| &\le 1+2\epsilon , \\
  |w_2-w_0| &\le \frac{2\vv}{\sqrt{3}}(1+2\epsilon) , \\
  |w_2-w_1| &\le \frac{\vv}{\sqrt{3}}(\sqrt{6}+4\epsilon)  , 
\end{align*}
and
\[
 \frac{|(w_2-w_1)(w_1-w_0)(w_2-w_0)|}{(\im((\overline{w_1-w_0})(w_2-w_0)))^2}
  \le \frac{7}{3} .
\]
\end{lemma}
\begin{proof}
Let $w_0 = \epsilon_1 + \rmi \epsilon_2$,
$w_1 = 1 + \epsilon_3 + \rmi \epsilon_4$,
$w_2 = \uu + \epsilon_5 + \rmi (\vv + \epsilon_6)$.
Then we have
\begin{align*}
 |w_1 - w_0|
 &\le |w_1 - 1| + |1- 0| + |0-w_0| 
 \le 1+2\epsilon, \\
 \frac{|w_2 - w_0|}{\vv}
  &\le \frac{|w_2-w|}{\vv} + \frac{|w-0|}{\vv} + \frac{|0-w_0|}{\vv} 
 \le \frac{2\epsilon}{\sqrt{3}} + \frac{2}{\sqrt{3}}
 + \frac{2\epsilon}{\sqrt{3}}
  = \frac{2}{\sqrt{3}}(1+2\epsilon), \\
 \frac{|w_2 - w_1|}{\vv}
  &\le \frac{|w_2-w|}{\vv} + \frac{|w-1|}{\vv} + \frac{|1-w_1|}{\vv}
  \le   \frac{2\epsilon}{\sqrt{3}} + \sqrt{2} + \frac{2\epsilon}{\sqrt{3}}
  = \frac{\sqrt{6}+4\epsilon}{\sqrt{3}},
\end{align*}
\begin{align*}
  \frac{\im((\overline{w_1-w_0})(w_2-w_0))}{\vv}
  &= \frac{(1+\epsilon_3-\epsilon_1)(\vv+\epsilon_6-\epsilon_2)
    -(\uu+\epsilon_5-\epsilon_1)(\epsilon_4-\epsilon_2)}{\vv} \\
 &\ge (1-2\epsilon) \left(1-\frac{4\epsilon}{\sqrt{3}}\right)
   - \frac{4\epsilon}{\sqrt{3}} \left( \half+2\epsilon \right) \\
 &= 1-(2+2\sqrt{3})\epsilon,
\end{align*}
and
\begin{align*}
\frac{|(w_2-w_1)(w_1-w_0)(w_2-w_0)|}{(\im((\overline{w_1-w_0})(w_2-w_0)))^2} 
&\le \frac{2(1+2\epsilon)^2(\sqrt{6}+4\epsilon)}{3(1-(2+2\sqrt{3})\epsilon)^2}
\le \frac{7}{3}
\end{align*}
for $0<\epsilon < \frac{1}{50}$.
\end{proof}

%%%%%
\begin{lemma}
\label{53}
Let $\ww = \uu + \vv\rmi \in \fdom_0$.  Let $w_0 \in U(0,\epsilon)$, $w_1
\in U(1,\epsilon)$, $w_2 \in U(\ww, \epsilon)$.  Suppose that
$\epsilon \le \frac{1}{45+10\vv}$.  Let $w_c$ be the circumcenter of
$\triangle(w_0,w_1,w_2)$.  Then we have $\frac{1}{4} \le \re(w_c) \le
\frac{3}{4}$, $\frac{\vv}{4} \le \im(w_c) \le \frac{3\vv}{4}$.
\end{lemma}
\begin{proof}
The circumcenter $w_{c,0}$ of $\triangle(0,1,\ww)$ is written as
\[ w_{c,0} = \half + \frac{\uu^2-\uu+\vv^2}{2\vv}\rmi,
\]
where we have
\[ \frac{\vv}{3} \le \frac{\vv}{2(1+\uu)} \le \frac{\uu^2-\uu+\vv^2}{2\vv} \le \frac{\vv}{2}
\]
for $\ww \in D$.
Let $w_{c,1}$ be the circumcenter of $\triangle(0,w_1,w_2)$,
and $w_{c,2}$ the circumcenter of $\triangle(0,1,w_2)$.
By Lemmas \ref{51} and \ref{52}, we have
\begin{align*}
|w_{c,2} - w_{c,0}| &\le \frac{|1-0|}{2} \cdot \frac{7}{3} |w_2 - \ww|
 \le \frac{7\epsilon}{6} , \\
|w_{c,1} - w_{c,2}| &\le \frac{|w_2-0|}{2} \cdot \frac{7}{3} |w_1-1| 
 \le  \frac{2\vv}{\sqrt{3}} (1+\epsilon)\cdot \frac{7\epsilon}{6} , \\
|w_c - w_{c,1}| &\le \frac{|w_2-w_1|}{2} \cdot \frac{7}{3} |w_0 - 0|
 \le  \frac{\vv}{\sqrt{3}} (\sqrt{6}+4\epsilon) \cdot \frac{7\epsilon}{6}  ,
\end{align*}
which imply that
\[
|w_c - w_{c,0}| 
\le \frac{7\epsilon}{6}
   \left(1+\frac{2\vv(1+\epsilon)}{\sqrt{3}}
       + \frac{v(\sqrt{6}+4\epsilon)}{\sqrt{3}} \right) 
\le \min\Set{ \frac{\vv}{12}, \frac{1}{4} }
\]
for $\epsilon < \frac{1}{45+10v}$.
This completes the proof.
\end{proof}

A mapping $\varphi : \cx \to \cx$ is called an $\epsilon$-perturbation
if it is a homeomorphisms and $|\varphi(\zeta) -\zeta| < \epsilon$ for
any $\zeta \in \cx$.

%%%%%
\begin{lemma}
\label{54}
Let $\ww = \uu + \vv\rmi \in \fdom_0$.  Let $\Lambda = \zahl + \ww \zahl$ be
a linear lattice.  Let $\varphi : \cx \to \cx$ be an
$\epsilon$-perturbation, where we suppose that $\epsilon \le
\frac{1}{45+10\vv}$.  Then, for any $\lambda \in \Lambda \setminus
\set{0,1,\ww, 1+\ww}$, $\varphi(\lambda)$ is out of the circumdisk $K$
of $\triangle(\varphi(0), \varphi(1), \varphi(\ww))$.
\end{lemma}

\begin{proof}
Denote by $\varphi(0) = \epsilon_1 + \epsilon_2 \rmi$, $\varphi(1) = 1
+ \epsilon_3 + \epsilon_4 \rmi$, $\varphi(\ww) = \uu + \vv\rmi +
\epsilon_5 + \epsilon_6 \rmi$.  Let $w_c = \sigma + \rho\rmi$ be the
circumcenter of $\triangle(\varphi(0), \varphi(1), \varphi(\ww))$.  By
Lemma~\ref{53}, we have $\frac{1}{4} \le \sigma \le \frac{3}{4}$,
$\frac{\vv}{4} \le \rho \le \frac{3\vv}{4}$.  Let $r$ be the radius
of $K$, where we have
\[ r^2 = (\epsilon_1 - \sigma)^2 + (\epsilon_2 - \rho)^2
 = (1+\epsilon_3 - \sigma)^2 + (\epsilon_4 - \rho)^2
 = (\uu + \epsilon_5 - \sigma)^2 + (\vv + \epsilon_6 - \rho)^2.
\]

Let $(j,k) \in \zahl^2 \setminus \set{(0,0), (1,0), (0,1), (1,1)}$.
We are going to show that $|\varphi(j+k\ww) - w_c| > r$.  Denote by
$\varphi(j+k\ww) = j + k\ww + \epsilon_7 + \epsilon_8\rmi$, where $\epsilon_7^2 +
\epsilon_8^2 < \epsilon^2$.

Case (i).  If $j\ge2$ and $k=0$, then
\begin{align*}
|\varphi(j) - w_c|^2 - r^2
&= (j+\epsilon_7 - \sigma)^2 + (\epsilon_8 - \rho)^2
- (1 + \epsilon_3 - \sigma)^2 - (\epsilon_4 - \rho)^2 \\
&= (1+j -2\sigma+\epsilon_7+\epsilon_3)(j-1+\epsilon_7-\epsilon_3)
  + (-2\rho + \epsilon_8 + \epsilon_4)(\epsilon_8-\epsilon_4) \\
 &\ge (\frac{3}{2} - 2\epsilon)(1-2\epsilon)
  - (\frac{3\vv}{2}+2\epsilon) 2\epsilon
> 0.
\end{align*}

Case (ii).  If $j\le-1$ and $k=0$, then
\begin{align*}
|\varphi(j) - w_c|^2 - r^2
&= (j+\epsilon_7 - \sigma)^2 + (\epsilon_8 - \rho)^2
- (\epsilon_1 - \sigma)^2 - (\epsilon_2 - \rho)^2 \\
&= (j -2\sigma+\epsilon_7+\epsilon_1)(j+\epsilon_7-\epsilon_1)
  + (-2\rho + \epsilon_8 + \epsilon_2)(\epsilon_8-\epsilon_2) \\
  &\ge (\frac{3}{2} - 2\epsilon)(1-2\epsilon)
  - (\frac{3\vv}{2}+2\epsilon) 2\epsilon
> 0.
\end{align*}

Case (iii).  If $j\in\zahl\setminus\set{0,1}$ and $k=1$, then
\begin{align*}
&|\varphi(j+\ww) - w_c|^2 - r^2 \\
&= (j+\uu + \epsilon_7 - \sigma)^2 + (\vv+\epsilon_8 - \rho)^2
  - (\uu + \epsilon_5 - \sigma)^2 - (\vv+\epsilon_6 - \rho)^2 \\
&= (j +2\uu-2\sigma+\epsilon_7+\epsilon_5)(j+\epsilon_7-\epsilon_5)
  + (2\vv-2\rho + \epsilon_8 + \epsilon_6)(\epsilon_8-\epsilon_6) \\
&\ge (\frac{1}{2} - 2\epsilon)(2-2\epsilon)
  - (\frac{3\vv}{2}+2\epsilon) 2\epsilon
> 0.
\end{align*}

Case (iv). If $k\ge2$, then
\begin{align*}
& |\varphi(j+k\ww) - w_c|^2 - r^2 \\
&= (j+k\uu + \epsilon_7 - \sigma)^2 + (k\vv+\epsilon_8 - \rho)^2
    - (\uu + \epsilon_5 - \sigma)^2 - (\vv+\epsilon_6 - \rho)^2 \\
&= (j +(k+1)\uu-2\sigma+\epsilon_7+\epsilon_5)(j+(k-1)\uu+\epsilon_7-\epsilon_5) \\
& \qquad  + ((k+1)\vv-2\rho + \epsilon_8 + \epsilon_6)((k-1)\vv+\epsilon_8-\epsilon_6) \\
&\ge - (\uu-\sigma+\epsilon_5)^2 + (\frac{3\vv}{2} - 2\epsilon)(\vv-2\epsilon) \\
&\ge - (\frac{3}{4}+\epsilon)^2 + (\frac{3\vv}{2} - 2\epsilon)(\vv-2\epsilon) 
> 0,
\end{align*}
where we used the inequality $(t-\mu)(t-\nu) \ge - \frac{(\mu-\nu)^2}{4}$.

Case (v).  If $k\le-1$, then
\begin{align*}
& |\varphi(j+k\ww) - w_c|^2 - r^2 \\
&= (j+k\uu+\epsilon_7 - \sigma)^2 + (k\vv+\epsilon_8 - \rho)^2
   - (\epsilon_1 - \sigma)^2 - (\epsilon_2 - \rho)^2 \\
&= (j+k\uu -2\sigma+\epsilon_7+\epsilon_1)(j+k\uu+\epsilon_7-\epsilon_1)
  + (k\vv-2\rho + \epsilon_8 + \epsilon_2)(k\vv+\epsilon_8-\epsilon_2) \\
&\ge - (\sigma - \epsilon)^2 + (\frac{3\vv}{2}-2\epsilon)(\vv- 2\epsilon) \\
&\ge - (\frac{3}{4} - \epsilon)^2 + (\frac{3\vv}{2}-2\epsilon)(\vv- 2\epsilon) > 0.
\end{align*}
\end{proof}

%%%%%
\begin{lemma}
Let $\ww = \uu + \vv\rmi \in \fdom_0$. 
% Suppose that $-\half \le \uu \le \half$, $|\ww|\ge1$.
Let $\Lambda = \zahl + \ww \zahl$ be a linear lattice.  Let $\varphi :
\cx \to \cx$ be an $\epsilon$-perturbation, where we suppose that
$\epsilon \le \frac{1}{45+10\vv}$.  If the Voronoi cell $V(\varphi(0),
\varphi(\Lambda))$ is adjacent to $V(\varphi(j+k\ww),
\varphi(\Lambda))$, then $|j|\le 1$ and $|k| \le 1$.
\end{lemma}

\begin{proof}
For distinct $w_0, w_1, w_2 \in \cx$, 
denote the closed circumdisk of $\triangle(w_0, w_1, w_2)$
by $\KK(w_0,w_1,w_2)$.
Let $(j,k) \in \zahl^2$, and suppose that $|j|\ge2$ or $|k|\ge2$.
If $0 \le \uu \le \half$, the union of the closed disks
\begin{align}
K &:=\KK(\varphi(0),\varphi(1), \varphi(\ww))
 \cup \KK(\varphi(0), \varphi(\ww), \varphi(-1+\ww)) 
 \notag \\
 &\cup \KK(\varphi(0), \varphi(-1+\ww), \varphi(-1))  
  \cup \KK(\varphi(0), \varphi(-1), \varphi(-\ww))
 \notag \\
 &\cup \KK(\varphi(0), \varphi(-\ww), \varphi(1-\ww))
 \cup \KK(\varphi(0), \varphi(1-\ww), \varphi(1))
\label{55}
\end{align}
is a closed neighborhood of $\varphi(0)$.  Lemma~\ref{54} applies to
the six circumdisks in (\ref{55}), and we obtain $\varphi(j+k\ww)
\not\in K$.  This implies that the Voronoi cell $V(\varphi(0),
\varphi(\Lambda))$ is not adjacent to $V(\varphi(j+k\ww),
\varphi(\Lambda))$.

If $-\half \le \uu \le 0$, the union of the closed disks
\begin{align}
K &=\KK(\varphi(0),\varphi(1), \varphi(1+\ww))
 \cup \KK(\varphi(0), \varphi(1+\ww), \varphi(-1)) 
 \notag \\
 &\cup \KK(\varphi(0), \varphi(1), \varphi(-1)) 
 \cup \KK(\varphi(0), \varphi(-1), \varphi(-1-\ww))
 \notag \\
 &\cup \KK(\varphi(0), \varphi(-1-\ww), \varphi(-\ww))
 \cup \KK(\varphi(0), \varphi(-\ww), \varphi(1))
\label{56}
\end{align}
is a closed neighborhood of $\varphi(0)$.  Since $K \not\ni
\varphi(j+k\ww)$, $V(\varphi(0), \varphi(\Lambda))$ is not adjacent to
$V(\varphi(j+k\ww), \varphi(\Lambda))$.
\end{proof}

Let $R_2 > R_1 > \epsilon > 0$.  Let $U \subset \cx$ be a neighborhood
of the origin such that $U(0,R_2) \subset U$.  A mapping $\varphi : U
\to \cx$ is called a local $(\epsilon, R_1, R_2)$-perturbation if it
is a homeomorphism onto its image, and $|\varphi(\zeta) - \zeta| <
\epsilon$ for $|\zeta| \le R_2$, and $|\varphi(\zeta)| \ge R_1$ for
$|\zeta|\ge R_2$.

%%%%%
\begin{lemma}
\label{60}
Let $\ww = \uu + \vv\rmi \in \fdom_0$.
%  Suppose that $-\half \le \uu \le \half$, $|\ww|\ge1$.
Let $\Lambda = \zahl + \ww \zahl$ be a linear lattice.  Let $\varphi :
\cx \to \cx$ be a local $(\epsilon, 3+3\vv, 4+4\vv)$-perturbation,
where we suppose that $\epsilon \le \frac{1}{45+10\vv}$.  If
$V(\varphi(0), \varphi(\Lambda))$ is adjacent to $V(\varphi(j+k\ww),
\varphi(\Lambda))$, then $|j|\le 1$ and $|k| \le 1$.
\end{lemma}
\begin{proof}
Let $K$ be a closed neighborhood of $\varphi(0)$, defined in
(\ref{55}) or (\ref{56}).  Then $K \subset U(0, 2+2\vv)$.  If $j+k\ww
\in \Lambda$ and $|j+k\ww| \le U(0,4+4\vv)$, then Lemma~\ref{54}
implies that $V(\varphi(0), \varphi(\Lambda))$ is not adjacent to
$V(\varphi(j+k\ww), \varphi(\Lambda))$.  If $j+k\ww \in \Lambda$ and
$|j+k\ww| \ge U(0,4+4\vv)$, then $|\varphi(j+k\ww)| \ge 3+3\vv$, so
$\varphi(j+k\ww) \not\in K$, and $V(\varphi(0), \varphi(\Lambda))$ is
not adjacent to $V(\varphi(j+k\ww), \varphi(\Lambda))$.
\end{proof}

%%%%%
\begin{lemma}
\label{57}
Let $\ww = \uu + \vv\rmi \in \fdom_0$.
%  Suppose that $-\half \le \uu \le \half$, $|\ww|\ge1$.
Let $\Lambda = \zahl + \ww \zahl$ be a linear lattice.  Let $\varphi :
\cx \to \cx$ be a local $(\epsilon, 3+3\vv, 4+4\vv)$-perturbation, and
suppose that $\epsilon \le \frac{1}{45+10\vv}$.  Then $V(\varphi(0),
\varphi(\Lambda)) \subset U(0, \frac{3}{4}(1+v))$.
\end{lemma}
\begin{proof}
Denote the circumcenter of $\triangle(w_0, w_1, w_2)$ by $P(w_0, w_1,
w_2)$.  The Voronoi cell $V(\varphi(0), \varphi(\Lambda))$ is a subset
of the quadrilateral with the vertices
$P(\varphi(0),\varphi(1), \varphi(\ww))$,
$P(\varphi(0), \varphi(\ww), \varphi(-1))$,
$P(\varphi(0), \varphi(-1), \varphi(-\ww))$, and 
$P(\varphi(0), \varphi(-\ww), \varphi(1))$,
which are all contained in the region
$\set{\zeta : -\frac{3}{4} \le \re(\zeta) \le \frac{3}{4},
  -\frac{3v}{4} \le \im(\zeta)\le \frac{3v}{4}} 
  \subset U(0, \frac{3}{4}(1+v))$ by Lemma~\ref{53}.
\end{proof}

%%%%%%%%

Let $d>0$.  Let $W(d) = \set{\zeta \in \cx : -d \le \im(\zeta) \le d}$
be a strip of width $2d$.  Let $H(w) = \set{\zeta \in \cx: |\zeta| \le
  |\zeta-w|}$, $w \in \punccx := \cx \setminus \set{0}$, be a half
plane.  The symmetric difference of two subsets $A,B \subset\cx$ is
denoted by $A \triangle B = (A\setminus B) \cup (B \setminus A)$.  The
area of a subset $A \subset\cx$ is denoted by $|A|$.

\begin{lemma}
  \label{8}
  Let $d,\epsilon>0$, and $w =\uu+ \rmi \vv \in\cx$.  Suppose
  that $|w-1| < \epsilon < \frac{1}{2}$.  Then $|W(d) \cap(H(1)
  \;\triangle\;H(w))| < 4\epsilon d(2d+1)$.
\end{lemma}
\begin{proof}
The perpendicular bisector of the line segment $\ell(0,w)$ intersect
  the horizontal line $\set{\zeta\in\cx: \im(\zeta)=d}$ at the point
  $\frac{w}{2} + \frac{\rmi w}{\uu}(d-\frac{\vv}{2})
   = \frac{\uu}{2} - \frac{\vv}{\uu}(d-\frac{\vv}{2}) + \rmi d$.
  Since $|\uu- 1|<\epsilon$ and $|\vv| < \epsilon$, we have
\[ \left| \frac{\uu}{2} - \frac{\vv}{\uu}(d-\frac{\vv}{2}) - \half \right|
  < \frac{\epsilon}{2} + {2\epsilon} \left( d+\frac{1}{4} \right)
  = \epsilon({2d}+1).
\]
  So we have $W(d)\cap(H(1) \;\triangle\;H(w)) \subset L$, where
  $L:=\set{x+\rmi y: |x-\frac{1}{2}| \le \epsilon({2d}+1),
    |y| \le d}$.  The area of the rectangle $L$ is $|L| =
  2\epsilon({2d}+1) \cdot 2d$.
\end{proof}

\begin{proposition}
\label{66}
Let $\ww = \uu + \vv\rmi \in \fdom_0$.
%  Suppose that $-\half \le \uu \le \half$, $|\ww|\ge1$.
Let $\Lambda = \zahl + \ww \zahl$ be a linear lattice.
Let $\varphi : \cx \to \cx$ be 
a local $(\epsilon, 3+3\vv, 4+4\vv)$-perturbation.
Suppose that $\varphi(0)=0$ and $\epsilon \le \frac{1}{45+10\vv}$.
Then 
\[ |V(0, \Lambda) \;\triangle\; V(0, \varphi(\Lambda))| 
\le 12 \epsilon(1+v)(5+3v).
\]
\end{proposition}
\begin{proof}
Let $N_1 = \set{\pm 1, \pm \ww, \pm (1+w), \pm(1-\ww)}$.
By Lemmas \ref{60} and \ref{57}, we have
\begin{align*}
V(0, \Lambda) &= \bigcap_{\lambda\in N_1} H(\lambda) \subset U(0, R), \\
V(0, \varphi(\Lambda)) &= \bigcap_{\lambda\in N_1} H(\varphi(\lambda)) \subset U(0, R)
\end{align*}
where $R:= \frac{3}{4}(1+v)$.  So
\[ V(0,\Lambda) \;\triangle\; V(0,\varphi(\Lambda))
 \subset \bigcup_{\lambda \in N_1}((H(\lambda) \;\triangle\; H(\varphi(\lambda))) \cap U(0,R).
\]
By Lemma~\ref{8}, we have
\[ 
 |(H(\lambda) \;\triangle\; H(\varphi(\lambda))) \cap U(0, R)|
 \le 4\frac{\epsilon}{|\lambda|} \frac{R}{|\lambda|} \left(\frac{2R}{|\lambda|}+1 \right) |\lambda|^2
 \le  4\epsilon R (2R +1).
\]
So we obtain
\[ 
   |V(0,\Lambda) \;\triangle\; V(0,\varphi(\Lambda))|
   \le 32 \epsilon R(2R+1) 
    = 12 \epsilon(1+v)(5+3v).
\]
\end{proof}

%%%%%%%%%%%%%%%%%%%%%%%%%%%%%%%%%%%%%%%%%%%%%%%%%%%%%%%%%%
\section{Area convergence of Voronoi cells}

In this section, an Archimedean spiral lattice is parameterized, normalized, and then linearized to show the area convergence of Voronoi cells.

Let $\alpha>0$, $\theta \in \real$, $\mu \ge 1$.  The set
\[
  S_\mu := \set{z_{\mu+j} :=
  (\mu+j)^\alpha \rme^{(\mu+j)\theta\rmi} : j\in\zahl, j> -\mu}
\cup \set{0}
\]
is called a parameterized Archimedean spiral lattice with an exponent
$\alpha$.  Note that $S_{\mu} = S_{\mu+1}$
as a point set.  So the $j$-th site $z_{\mu+j}$ in $S_\mu$ can be
regarded as the $0$-th site in $S_{\mu+j}$.  The Voronoi cell of
$z_{\mu+j}$ in $S_\mu$ is given as
\[
  V(z_{\mu+j}, S_\mu)
:= \set{\zeta \in \cx:  |\zeta - z_{\mu+j}| \le |\zeta - z'|,
 \forall z' \in S_\mu}.
\]

The point set 
\[
 \Sigma_\mu 
 := \Set{\tildez_{\mu,j}:= \left(1+ \frac{j}{\mu}\right)^\alpha
  \rme^{j\theta\rmi} : j\in\zahl, j>-\mu} \cup \set{0}  
  = \mu^{-\alpha} \rme^{- \mu \theta \rmi} S_\mu 
\]
is
called a normalized spiral lattice.  The Voronoi cell of
$\tildez_{\mu,j}$ in $\Sigma_\mu$ is given as
\[
V(\tildez_{\mu,j},\Sigma_\mu)
:= \set{\zeta \in \cx:  |\zeta - \tildez_{\mu,j}| \le |\zeta - \tildez'|,
 \forall \tildez' \in \Sigma_\mu}.
\]
We have $V(z_{\mu+j}, S_\mu) = \mu^\alpha \rme^{\mu \theta \rmi}
V(z_{\mu,j},\Sigma_\mu)$ and $|V(z_{\mu+j}, S_\mu)| = \mu^{2\alpha}
|V(z_{\mu,j}, \Sigma_\mu)|$.  In particular, by taking $j=0$ we have
\begin{equation}
  \label{13}
  |V(z_{\mu}, S_\mu)| = \mu^{2\alpha} |V(1,\Sigma_\mu)|.
\end{equation}

The linear lattice
\[ \Lambda_\mu := \lambda_\mu \zahl + 2\pi \rmi \zahl,
\quad \lambda_\mu := \frac{\alpha}{\mu} + 2\pi \fracpart{\frac{\theta}{2\pi}} \rmi
\]
is called a linearization of $S_\mu$.  Let $\lambda_{\mu,j}:=\frac{j
  \alpha}{\mu} + 2\pi \fracpart{\frac{j\theta}{2\pi}} \rmi \in
\Lambda_\mu$, $j>-\mu$.  The Voronoi cell of $\lambda_{\mu,j}$ in
$\Lambda_\mu$ is given as
\[ V(\lambda_{\mu,j},\Lambda_\mu) = 
 \set{\zeta \in \cx:  |\zeta - \lambda_{\mu,j}| \le |\zeta - \lambda'|,
 \forall \lambda' \in \Lambda_\mu}.
\]

Let $W=\set{s+\rmi t : s>-\mu, -\pi<t<\pi}$.  Define a mapping
$\phi : W \to \cx$ by 
\[ \phi(s+\rmi t) = \left( 1+\frac{s}{\mu} \right)^\alpha \rme^{\rmi t}.
\] 
 It is a homeomorphism of
$W$ onto $\phi(W) = \cx \setminus \set{\zeta \in \real: \zeta \le 0}$.
We have 
\[ \phi(W \cap \Lambda_\mu) \cup \set{0} = \Sigma_\mu
\]
if $\frac{\theta}{2\pi}$ is irrational.  By Taylor's theorem, there
exists a constant $C=C_{\alpha}>0$ which depends only on $\alpha$,
such that
\[ |\phi(s+\rmi t)-(1+\frac{\alpha s}{\mu} +\rmi t)| \le C_\alpha (s^2 + t^2)
\]
for $s + t\rmi \in U(0,1)$.

\bigskip

Suppose that $\frac{\theta}{2\pi}$ is irrational.  We consider the
continued fraction expansion of $x=\frac{\theta}{2\pi}$ as in
(\ref{25}).  We assume that
\begin{equation}
\label{67}
     a_i \le M_1,
\end{equation}
for some constant $M_1>0$.
%That is, the partial quotients of the continued fraction expansion of $x$ are bounded.
%We assume (\ref{67}).  
An integer $j>0$ is called a parastichy number of the linear lattice
$\Lambda_\mu$ if $V(0,\Lambda_\mu)$ is edge-adjacent to
$V(\lambda_{\mu,j},\Lambda_\mu)$.  Let
\begin{equation}
  \label{16}
  \mu_{i,k}
= \frac{\alpha}{2\pi} \sqrt{ - \frac{q_i q_{i,k}}
{\fracpart{\frac{q_i \theta}{2\pi}}
  \fracpart{\frac{q_{i,k}\theta}{2\pi}}}}
\end{equation}
for $i\le1$ and $0 \le k \le a_{i+1}$.  Note that
$\mu_{i,k}<\mu_{i,k+1}$ for $0\le k < a_{i+1}$, and $\mu_{i,a_{i+1}} =
\mu_{i+1,0}$.  If $\mu=\mu_{i,k}$, we have
\[
 q_{i} q_{i,k} \left(\frac{\alpha}{2\pi\mu}\right)^2 
+ \fracpart{\frac{q_i\theta}{2\pi}}
\fracpart{\frac{q_{i,k}\theta}{2\pi}} = 0,
\]
and so $|\angle(\lambda_{\mu,q_i},0,\lambda_{\mu,q_{i,k}})| =
\frac{\pi}{2}$.

Let $i\ge1$, $0\le k < a_{i+1}$.  Lemma~\ref{58} implies that if $\mu
= \mu_{i,k}$, then the Voronoi cell $V(0, \Lambda_\mu)$ is a rectangle
with two parastichy numbers $q_i, q_{i,k}$.  If $\mu_{i,k} < \mu <
\mu_{i,k+1}$, then $V(0,\Lambda_\mu)$ is a hexagon with three
parastichy numbers $q_i, q_{i,k}, q_{i,k+1}$.

Suppose that $\mu_{i,k} \le \mu < \mu_{i,k+1}$.  Then there exist
parastichy numbers $m,n \in \set{q_i, q_{i,k}, q_{i,k+1}}$
of $\Lambda_\mu$ such that
\begin{equation}
  \label{21}
  \lambda_{\mu,n}/\lambda_{\mu,m} \in \fdom \cup (-\fdom),
\end{equation}
where $\fdom = \set{w=u+v\rmi \in \cx : -\half \le u \le \half, v \ge
 \frac{\sqrt{3}}{2}, |w|\ge1}$.  Let
\[ \delta_\mu := |\lambda_{\mu,m}| 
 = \min\set{ |\lambda_{\mu,q_i}|, |\lambda_{\mu,q_{i,k}}|, |\lambda_{\mu,q_{i,k+1}}|}
\]
be the length of the shortest nonzero vector in $\Lambda_\mu$.
Let $v_\mu := |\im(\lambda_{\mu,n} / \lambda_{\mu,m})|$.
Then we have $|V(0,\Lambda_\mu)| = v_\mu \delta_\mu^2$,
$v_\mu \ge \frac{\sqrt{3}}{2}$, 
and
\[ \left( \frac{\delta_\mu}{2\pi} \right)^2
\le \frac{2}{\sqrt{3}} \frac{\alpha}{2\pi\mu}
\]
as in (\ref{72}).
Lemma~\ref{68} implies that
$v_\mu \le M_2 := 1 + \half M_1$.

Let
\[ 
 \epsilon_\mu := C_\alpha (4+4M_2)^2 \sqrt{ \frac{4\pi\alpha}{\sqrt{3} \mu}}.
\]

\begin{lemma}
\label{70}
Suppose (\ref{67}) and 
\begin{equation}
\label{63}
 \mu \ge \frac{4\pi\alpha}{\sqrt{3}} (16C_\alpha(1+M_2))^2. 
\end{equation}
Then the mapping $-1+\phi : W \to \cx$ is a local
$(\epsilon_\mu\delta_\mu, (3+3v_\mu)\delta_\mu,
(4+4v_\mu)\delta_\mu)$-perturbation.
\end{lemma}
\begin{proof}
For $|\zeta| \le (4 + 4 v_\mu)\delta_\mu$, we have
\[
 |\phi(\zeta)-1-\zeta|
\le C_\alpha |\zeta|^2 \le C_\alpha (4+4 v_\mu)^2\delta_\mu^2 
 \le \epsilon_\mu \delta_\mu.
\]
For $|\zeta| = (4+4 v_\mu)\delta_\mu$, we have
\begin{align*}
 |-1+\phi(\zeta)| 
 &\ge |\zeta| - |\phi(\zeta)-1 - \zeta|  \\
 &\ge (4+4 v_\mu)\delta_\mu - C_\alpha(4+4 v_\mu)^2\delta_\mu^2  \\
 &\ge (3 + 3v_\mu)\delta_\mu
\end{align*}
by (\ref{63}).
\end{proof}

\begin{lemma}
\label{73}
Suppose (\ref{67}), (\ref{63}),
and
\begin{equation}
\label{69}
\mu \ge  \frac{4\pi\alpha}{\sqrt{3}} C_\alpha^2 (45+10M_2)^2 (4+4M_2)^4.
\end{equation}
Then there exists a contant $C>0$ depending only on $\alpha$, such that
\begin{equation}
\label{24}
  \frac{|V(0,\Lambda_\mu) \;\triangle\; V(0,\phi(W \cap \Lambda_\mu)-1)|}
       {|V(0,\Lambda_\mu)|}
  \le \frac{C}{\sqrt{\mu}}.
\end{equation}
\end{lemma}
\begin{proof}
By (\ref{69}), we have $\epsilon_\mu \le \frac{1}{45+10v_\mu}$.
Proposition~\ref{66} and Lemma~\ref{70} imply that 
\[
 \frac{|V(0,\Lambda_\mu) \;\triangle\; V(0,\phi(W \cap \Lambda_\mu)-1)|}      {\delta_\mu^2}
 \le 12 \epsilon_\mu (1+v_\mu)(5+3v_\mu).
\]
Since $|V(0,\Lambda_\mu)| = v_\mu \delta_\mu^2$, we have
\begin{align*}
 & \frac{|V(0,\Lambda_\mu) \;\triangle\; V(0,\phi(W \cap \Lambda_\mu)-1)|}
 {|V(0,\Lambda_\mu)|} \\
% =  \frac{|V(0,\Lambda_\mu) \;\triangle\; V(0,\phi(W \cap \Lambda_\mu)-1)|}
%{v_\mu \delta_\mu^2}
 & \le 12 \epsilon_\mu \frac{(1+v_\mu)(5+3v_\mu)}{v_\mu} 
  \le 12 \epsilon_\mu (15+3v_\mu) 
 \le 12 \epsilon_\mu (15+3M_2) \\
 & = C \mu^{-1/2} ,
\end{align*}
where $C := 12C_\alpha (15+3M_2)(4+4M_2)^2
\sqrt{\frac{4\pi\alpha}{\sqrt{3}} }$.
\end{proof}

Lemma~\ref{20} implies that 
\begin{equation}
\label{59}
  \frac{|V(0,\Lambda_\mu)|}{(2\pi)^2}  = \frac{\alpha}{2\pi\mu}.
\end{equation}
Thus Lemma~\ref{73} reads
\begin{equation}
\label{64}
 |V(1,\Sigma_\mu)|
 = |V(0,\phi(W \cap \Lambda_\mu)-1)|
 = \frac{2\pi\alpha}{\mu} (1 + O(\mu^{-1/2})).
\end{equation}

\begin{theorem}
  \label{19}
  Suppose that the partial quotients $a_i$ in the continued
  fraction expansion of $\theta/2\pi$ is bounded.  Then
  \[ 
   |V(z_{\mu}, S_\mu)| = 2\pi\alpha \mu^{2\alpha-1}(1 + O(\mu^{-1/2})).
  \]
\end{theorem}
\begin{proof}
This follows from (\ref{13}) and (\ref{64}).
\end{proof}

Theorem~\ref{18} follows immediately from Theorem~\ref{19}.

%%%%%%%%%%%%%%%%%%%%%%%%%%%%%%%%%%%%%%%%%%%%%%%%%%%%%%%%%%%%%%%%
\section*{Acknowledgement} 
The authors thank Shigeki Akiyama and Qinghui Liu for helpful comments
and fruitful discussions.  This work was partially supported by JSPS
Kakenhi Grant Numbers JP15K05011, JP18K13452, and the Research
Institute for Mathematical Sciences, a Joint Usage/Research Center
located in Kyoto University.

%%%%%%%%%%%%%%%%%%%%%%%%%%%%%%%%%%%%%%%%%%%%%%%%%%%%%%%%%%%%%%%%
%\section*{References}

\end{document}